\documentclass[11pt,oneside,a4paper]{article}
\usepackage{syntonly}
\usepackage{geometry}
\usepackage{mathrsfs}
\usepackage{makeidx}
\makeindex
\usepackage{amsmath}
\usepackage{amssymb}
\usepackage{amsmath}
\usepackage{amsfonts}
\usepackage{amssymb}
\usepackage{esvect}
\usepackage{algorithm}
\usepackage{bm}
\usepackage{tikz}
\usepackage{cleveref}
\usepackage{graphicx}
\usepackage{epstopdf}

\newtheorem{theorem}{Theorem}

\newtheorem{corollary}{Corollary}

\newenvironment{proof}[1][Proof]{\begin{trivlist}
\item[\hskip \labelsep {\bfseries #1}]}{\end{trivlist}}

\newcommand{\qed}{\nobreak \ifvmode \relax \else
      \ifdim\lastskip<1.5em \hskip-\lastskip
      \hskip1.5em plus0em minus0.5em \fi \nobreak
      \vrule height0.30em width0.4em depth0.25em\fi}

 \author{Safari Mukeru\\
\footnotesize{\em Department of Decision Sciences}\\
\footnotesize{University of South Africa, P. O. Box 392, Pretoria, 0003. South Africa}\\
\footnotesize{e-mail:mukers@unisa.ac.za}}

\title{{A generalisation of Pisier homogeneous Banach algebra}}

\date{}

\begin{document}

\maketitle

\pagenumbering{arabic}

\begin{abstract}
In 1979 Pisier proved remarkably that a sequence of independent and identically distributed standard Gaussian random variables determines, via random Fourier series, a homogeneous Banach algebra $\mathscr{P}$ strictly contained in $C(\mathbb{T})$, the class of continuous functions on the unit circle $\mathbb{T}$ and strictly containing the classical Wiener algebra $\mathbb{A}(\mathbb{T})$, that is, $\mathbb{A}(\mathbb{T}) \subsetneqq \mathscr{P} \subsetneqq C(\mathbb{T}).$ This improved some previous results obtained by Zafran in solving a long-standing problem raised by Katznelson. In this paper we extend Pisier's result by showing that any probability measure on the unit circle defines a homogeneous Banach algebra contained in $C(\mathbb{T})$. Thus Pisier algebra is not an isolated object but rather an element in a large class of Pisier-type algebras. We consider the case of spectral measures of stationary sequences of Gaussian random variables and obtain a sufficient condition for the boundedness of the random Fourier series $\sum_{n\in \mathbb{Z}}\hat f(n) \,\xi_n \exp(2\pi i n t)$ in the general setting of dependent random variables $(\xi_n)$.

\end{abstract}

\section{Introduction}
Pisier \cite{Pisier} constructed a homogeneous Banach algebra $\mathscr{P}$ of continuous functions on the unit circle $\mathbb{T} = \mathbb{R}/\mathbb{Z}$ such that $\mathbb{A}(\mathbb{T}) \subsetneqq \mathscr{P} \subsetneqq C(\mathbb{T})$ where $C(\mathbb{T})$ is the Banach algebra of continuous functions on $\mathbb{T}$ and $\mathbb{A}(\mathbb{T})$ is the classical Winer algebra of continuous functions on $\mathbb{T}$ with absolutely convergent Fourier series. He further obtained that all 1-Lipschitzian functions operate on $\mathscr{P}$. To fully describe $\mathscr{P}$, we shall consider a fixed probability space $(\Omega, \mathscr{F}, \mathbb{P})$ and  a Gaussian Hilbert space $\mathscr{H}$, that is, a closed subspace of $L^2(\mathbb{P})$ whose elements are all Gaussian random variables.
Pisier algebra $\mathscr{P}$ is determined by a sequence $Z = (Z_n)_{n \in \mathbb{Z}}$ of independent standard Gaussian random variables. We recall that a Gaussian process on $\mathbb{T}$ is simply a family 
$\{X(t): t \in \mathbb{T}\}$ of random variables in $\mathscr{H}$. 
For $f\in C(\mathbb{T})$ and $t \in \mathbb{T}$, the random Fourier  series  $\sum_{n \in \mathbb{Z}} \hat f(n) Z_n e^{2 \pi n i t}$ converges obviously in $L^2(\mathbb{P})$ because $\sum_{n\in \mathbb{Z}} |\hat f(n)|^2 < \infty.$
The class $\mathscr{P}$ is defined as the class of functions $f \in C(\mathbb{T})$ such that the Gaussian process $S(f, Z) = \{S(f, Z, t): t\in \mathbb{T}\}$ consisting of the $L^2$-sums $S(f, Z, t) = \sum_{n \in \mathbb{Z}} \hat f(n) Z_n e^{2 \pi n i t}$ has bounded sample paths almost surely.  Equivalently $\mathscr{P}$ is the class of functions $f \in C(\mathbb{T})$ such that the Fourier random series $\sum_{n \in \mathbb{Z}} \hat f(n) Z_n e^{2 \pi n i t}$ converges in $C(\mathbb{T})$. Other equivalent characterisations of $\mathscr{P}$ can be obtained from the classical Billard theorem (see Cohen and Cuny \cite{Cohen_Cuny} and Kahane \cite[p 58]{Kahane_1985}). 

Given a Borel probability measure $\mu$ on $\mathbb{T}$ and a continuous function $f \in C(\mathbb{T})$, we shall associate a Gaussian process $X_f = \{X_f(t): t \in \mathbb{T}\}$ such that for all $t,s\in \mathbb{T}$,
     $$\mathbb{E}(|X_f(t) - X_f(s)|^2) = \int_{\mathbb{T}}|f(t+u) - f(s+u)|^2 d\mu(u)$$  
and consider the class $\mathscr{P}(\mu)$ of those continuous functions $f$ on $\mathbb{T}$ for which the process $X_f$ has bounded paths (almost surely) and set
$$     \| f\|_{\mathscr{P}(\mu)} = \mathbb{E}\left(\sup_{s,t,\in \mathbb{T}} |X_f(t) - X_f(s)|\right) + \|f\|_\infty.$$
Pisier algebra $\mathscr{P}$  is $\mathscr{P}(\mu)$ where $\mu$ is the Lebesgue measure on $\mathbb{T}$.
We shall obtain that $\mathscr{P}(\mu)$ is a homogeneous Banach algebra on which all 1-Lipschitz functions operate and further we shall obtain a sufficient condition for $\mathscr{P}(\mu) \subset \mathscr{P}$. 

As already indicated, Pisier algebra arose from the study of random Fourier series of i.i.d Gaussian random variables.  
We shall consider throughout a discrete-time complex Gaussian process $\xi = (\xi_n)_{n\in \mathbb{Z}}$  in the Gaussian Hilbert space $\mathscr{H}$ such that $\mathbb{E}(\xi_n) = 0$ and  $\mathbb{E}(|\xi_n|^2) = 1$ for all $n \in \mathbb{Z}$. We shall further assume that for all $n,m\in \mathbb{Z}$, $\mathbb{E}(\xi_n \overline{\xi_{m}})$ depends only on $n - m$ and $\mathbb{E}(\xi_n \xi_m) = 0$. This implies that the process $\xi =(\xi_n)_{n \in \mathbb{Z}}$ is stationary. Clearly for any $Z \in \mbox{span}(\xi)$, $Z$ is a complex Gaussian variable with distribution $f(z) =\exp(-|z|^2/\sigma^2)/(\pi \sigma^{2})$ where $\sigma^2 = \mathbb{E}(|Z|^2)$. 
The covariance function of $\xi$ is the function $\gamma: \mathbb{Z} \to \mathbb{C}$ defined by $\gamma(n) = \mathbb{E}(\xi_0 \overline{\xi_{n}}).$ Clearly $\gamma(-n) = \overline{\gamma(n)}$ and $\gamma(0) = 1$. 
Such sequence $(\xi_n)$ can be obtained by taking a sequence $(\zeta_n)$ of {\it real} standard Gaussian random variables such that $\mathbb{E}(\zeta_n \zeta_m)$ depends only on $(n-m)$ and thereafter take
   $\xi_n = (\zeta_n + i \zeta_n')/\sqrt{2}$ where $(\zeta_n')$ is an independent copy of $(\zeta_n)$.  
Since the covariance function $\gamma$ is positive semidefinite, the classical Bochner theorem implies that there exists a Borel probability measure $\mu$ on the unit circle $\mathbb{T}$ such that 
      \begin{eqnarray} \label{sde34rewd2we3}
\mathbb{E}(\xi_n \overline{\xi_{m}}) = \gamma(n-m) = \int_{\mathbb{T}} e^{2\pi i(n-m) t} d\mu(t),\,\,n,m\in \mathbb{Z}.
\end{eqnarray}
The measure $\mu$ is the spectral measure of the process $\xi$. If  $\mu$ is absolutely continuous with respect to the Lebesgue measure on $\mathbb{T}$, its density $\varphi$  is called the spectral density function of the process $\xi$.  In the corresponding Banach algebra $\mathscr{P}(\mu)$ (where $\mu$ is the spectral measure of $\xi$), it is easy to see that the {\it formal} random Fourier series  $\sum_{n \in \mathbb{Z}} \hat f(n) \xi_n e^{2 \pi i n t}$ is almost surely the {\it formal} Fourier series of the function $t \to X_f(t)$. This is true because the Gaussian sequence $(\hat f(n) \xi_n)_{n \in \mathbb{Z}}$ has the same distribution with the sequence of the Fourier coefficients $(\hat{X_f}(n))_{n\in \mathbb{Z}}$ of $X_f$ given by
   $$\hat{X_f}(n) = \int_{\mathbb{T}} X_f(t) e^{-2\pi i n t} dt.$$ 
If  $\mu$ has a  density $\varphi$ such that  $\varphi \in L^p(du)$ for some $1< p < \infty$, then for any $f \in \mathscr{P}(\mu)$ and $t\in \mathbb{T}$, the formal Fourier random series $\sum_{n \in \mathbb{Z}} \hat f(n) \xi_n e^{2\pi i n t}$ converges in $L^2(\mathbb{P})$. We shall obtain a sufficient condition on the covariance matrix of the sequence $(\xi_n)$ for the almost surely boundedness of the random Fourier series  $\sum_{n\in \mathbb{Z}} \hat f(n) \xi_n e^{2\pi n i t}$ that generalises a classical condition only known for i.i.d variables. This is  in particular applied to the classical fractional Gaussian noise. 

The study of random Fourier series $\sum_{n\in \mathbb{Z}} \hat f(n) \xi_n \exp(2\pi i n t)$ where the random variables $(\xi_n)$ are interdependent is important as it builds a bridge between the classical random series 
$\sum_{n\in \mathbb{Z}} \hat f(n) Z_n \exp(2\pi i n t)$ of independent variables $(Z_n)$ and the deterministic Fourier series $\sum_{n\in \mathbb{Z}} \hat f(n) \exp(2\pi i n t)$ which is the limit case where all the random variables $(\xi_n)$ are the same. For this limit case the problem of $L^2$-convergence becomes the problem of characterisation of the classical class $\mathscr{C}$ of functions $f \in C(\mathbb{T})$ with everywhere convergent Fourier series. (This is very complex since $\mathscr{C}$ is beyond the Borel hierarchy in $C(\mathbb{T})$ (see Ajtai and Kechris \cite{Ajtai_Kechris}).) 
This paper can be seen as a step toward the study of random Fourier series in the Banach space $C(\mathbb{T})$ in the presence of interdependent random variables which constitutes an important open problem.

\section{Banach algebras defined by probability measures}

Given a Borel probability measure $\mu$ on $\mathbb{T}$ and a $f \in C(\mathbb{T})$, we can associate to $f$  the pseudo-distance $d_f$ on $\mathbb{T}$ defined by
     $$d_f(t,s) = \left(\int_{\mathbb{T}} |f(t+u) - f(s+u)|^2 d\mu(u)\right)^{1/2},\,\,t,s \in \mathbb{T}.$$
Since the space $L^2(\mu, \mathbb{C})$ embeds in a Gaussian Hilbert space, there exists a Gaussian process $X_f = \{X_f(t): t \in \mathbb{T}\}$ such that for all $t,s\in \mathbb{T}$,
$$\mathbb{E}\left(X_f(t) \overline{X_f(s)}\right) = \int_{\mathbb{T}} f(t+u) \overline{f(s+u)} d\mu(u)$$ and hence
     $$\mathbb{E}(|X_f(t) - X_f(s)|^2) = d_f^2(t,s)\,\, \, \mbox{ and } \mathbb{E}(|X_f(t)|^2) = \int_{\mathbb{T}} |f(t+u)|^2 d\mu(u).$$  
More precisely, one may consider an orthonormal basis $(e_n)_{n \in A}$ of $L^2(\mu, \mathbb{C})$ (where $A$ is a subset of $\mathbb{Z}$), a sequence of independent and identically distributed (i.i.d) standard complex Gaussian random variables $(Z_n)$ and set for all $t \in \mathbb{T}$,
 $$X_f(t) = \sum_{n \in A} \langle f_t, e_n\rangle Z_n$$ where $f_t$ is the function defined on $\mathbb{T}$ by $f_t(u) = f(t+u)$ and 
 $$\langle g, e_n\rangle = \int_{\mathbb{T}} g(u) \overline{e_n(u)} d\mu(u),\,\,g \in L^2(\mu, \mathbb{C}).$$
(The convergence of the series $\sum_{n \in A} \langle f_t, e_n\rangle Z_n$ in $L^2(\mathbb{P})$ is guaranteed by the fact that $(Z_n)$ are independent $N(0,1)$ variables.) 
Clearly the process $X_f$ depends linearly on $f$. Let $\mathscr{P}(\mu)$ be the class of continuous functions $f$ on $\mathbb{T}$ such that the process $X_f$ has almost surely bounded paths and set
$$     \| f\|_{\mathscr{P}(\mu)} = \mathbb{E}\left(\sup_{s,t,\in \mathbb{T}} |X_f(t) - X_f(s)|\right) + \|f\|_\infty.$$
The classical Pisier algebra $\mathscr{P}$  is $\mathscr{P}(\mu)$ where $\mu$ is the Lebesgue measure on $\mathbb{T}$. We have the following result.
\begin{theorem}\label{th0101}
For every Borel probability measure $\mu$, the class $\mathscr{P}(\mu)$ is a homogeneous Banach algebra on which all 1-Lipschitz functions operate. 
\end{theorem}
\begin{proof}
The proof is a replication of Pisier's original argument in \cite{Pisier} and the following classical comparison principle: 
Let $X = \{X(t): t\in \mathbb{T}\}$ and $Y = \{Y(t): t\in \mathbb{T}\}$ be two real Gaussian processes. If  $$\mathbb{E}|Y(t) - Y(s)|^2 \leq \mathbb{E}|X(t) - X(s)|^2\,\,\,\,\mbox{ for all } t, s \in \mathbb{T},$$  
then
       $$\mathbb{E}\left(\sup_{s, t \in \mathbb{T}}|Y(t) - Y(s)|\right) \leq \mathbb{E}\left(\sup_{s, t \in \mathbb{T}}|X(t) -X(s)|\right).$$ In particular if $X$ has bounded paths almost surely then $Y$ has also bounded paths almost surely. 
Here $$\mathbb{E}\left(\sup\{|X(t)|: t \in \mathbb{T}\}\right) = \sup\left(\left\{\mathbb{E}\sup\{|X(t)|: t \in F\}: F \mbox{ finite subset of } \mathbb{T}\right\}\right). $$
The comparison principle extends easily to complex Gaussian processes. Indeed, write $X_1 = \Re(X)$, $X_2 = \Im(X)$ (both $X_1$ and $X_2$ are  real Gaussian processes) and set $X' = X_1 + X_2'$ where $X_2'$ is a real Gaussian process which is an independent copy of $X_2$. 
Then clearly 
$$\mathbb{E}\left(\sup_{t,s}|X(t) - X(s)|\right) \leq 2\mathbb{E}\left(\sup_{t,s}|X'(t) - X'(s)|\right),$$ $$\mathbb{E}\left(\sup_{t,s} |X'(t) - X'(s)|\right) \leq 2 \mathbb{E}\left(\sup_{t,s}|X(t) - X(s)|\right)$$ and
 $$\mathbb{E}|X'(t) - X'(s)|^2  = \mathbb{E}|X_1(t) - X_1(s)|^2 + \mathbb{E}|X_2'(t) - X_2'(s)|^2 = \mathbb{E}|X(t) - X(s)|^2.$$ 
 Similarly for $Y$, define $Y_1, Y_2, Y_2'$ and  write $Y' = Y_1 + Y_2'$ so that  
 $\mathbb{E}|Y'(t) - Y'(s)|^2  =  \mathbb{E}|Y(t) - Y(s)|^2$. Hence $\mathbb{E}|Y'(t) - Y'(s)|^2 \leq \mathbb{E}|X'(t) - X'(s)|^2$ which implies 
 $$\mathbb{E}\left(\sup_{t,s}|Y'(t) - Y'(s)|\right) \leq \mathbb{E}\left(\sup_{t,s}|X'(t) - X'(s)|\right).$$
Then 
\begin{eqnarray*}
\mathbb{E}\left(\sup_{t,s}|Y(t) - Y(s)|\right) &\leq &  2 \mathbb{E}\left(\sup_{t,s}|Y'(t) - Y'(s)|\right) \leq 
 2 \mathbb{E}\left(\sup_{t,s}|X'(t) - X'(s)|\right)\\
& \leq & 4 \mathbb{E}\left(\sup_{t,s}|X(t) - X(s)|\right).
\end{eqnarray*}
Using the comparison principle, it follows that for $f, g\in C(\mathbb{T})$, if 
  $$ \mathbb{E}(|X_f(t) - X_f(s))|^2 \leq \mathbb{E}(|X_g(t) - X_g(s)|^2, $$ 
then 
       $$\mathbb{E}\left(\sup_{s,t}|X_f(t) - X_f(s))|\right) \leq 4 \mathbb{E}\left(\sup_{s,t}|X_g(t) - X_g(s))|\right).$$ This implies that  if $g \in \mathscr{P}(\xi)$ then $f \in \mathscr{P}(\xi)$.      
It is now easy to use the same argument as in the paper by Pisier \cite{Pisier} to obtain the complete proof of Theorem \ref{th0101}. \hfill \qed
     
\end{proof}

In the case where the probability measure $\mu$ is absolutely continuous with density $\varphi$, that is,
$d\mu(u) = \varphi(u) du$, it is clear that if $\varphi$ is bounded, then  
$\mathscr{P}  \subset \mathscr{P}(\mu)$ where $\mathscr{P}$ is the original Pisier algebra (associated to the Lebesgue measure on $\mathbb{T}$). 

\begin{theorem} \label{referee_c1}
Let $\mu$ be a Borel probability measure on $\mathbb{T}$. If there is an interval $I \subset \mathbb{T}$ and a constant $C>0$ such that for any continuous function $g \geq 0$ on $\mathbb{T}$
 $$ \int_I g(u) du \leq C \int_I g(u) d\mu(u),$$ 
 then $\mathscr{P}(\mu) \subset \mathscr{P}.$ 
In particular if $\mu$ admits a density $\varphi$ that is continuous and strictly positive on some open interval $I$ of $\mathbb{T}$, then $\mathscr{P}(\mu) \subset \mathscr{P}.$      
\end{theorem}
\begin{proof}
Clearly there is a finite set of translations of $I$ that covers $\mathbb{T}$, that is, there exist $t_1, t_2, \ldots, t_N$ in $\mathbb{T}$ such that
$\mathbb{T} \subset (I + t_1) \cup \ldots \cup (I + t_N).$ Set $\mu' = (\sum_{k=1}^N \mu_j)/N$ with $\mu_j = \delta_{t_j} * \mu$. (Since  each $\mu_j$ is a probability measure then  $\mu'$ is  also a  probability measure.) 
 Then for any continuous function $g \geq 0$ on $\mathbb{T}$,
\begin{eqnarray*}
 \int_{\mathbb{T}} g(u) du  & \leq & \sum_{j=1}^N \int_{I + t_j} g(u) du = \sum_{j=1}^N \int_{I} g(u+t_j) du  \leq C \sum_{j=1}^N \int_{I} g(u+t_j) d\mu(u) \\
&  = &  C \sum_{j=1}^N \int_{I} g(u) d\mu_j(u) =  C N \int_{I} g(u) d\mu'(u)  \\
&\leq & C N \int_{\mathbb{T}} g(u) d\mu'(u).
\end{eqnarray*}
In particular
 for any $t,s \in \mathbb{T}$, $f\in C(\mathbb{T})$,
  $$\int_{\mathbb{T}}|f(t+u) - f(s+u)|^2 du \leq C N \int_{\mathbb{T}}|f(t+u) - f(s+u)|^2 d\mu'(u).$$ This implies by the comparison principle that $\mathscr{P}(\mu') \subset \mathscr{P}$. Moreover $\mathscr{P}(\mu') = \mathscr{P}(\mu)$. Indeed, for any $t_0 \in \mathbb{T}$, it is clear that $\mathscr{P}(\mu) = \mathscr{P}(\delta_{t_0} * \mu)$ since the corresponding processes $X_f$ are obtained from the other by a time translation by $t_0$. This implies that $\mathscr{P}(\mu') \subset \mathscr{P}(\mu)$. To see that $\mathscr{P}(\mu) \subset \mathscr{P}(\mu')$ also holds, consider independent processes $X_f^1, X_f^2, \ldots, X_f^N$ such that $X_f^j$ is the process associated to the measure $\delta_{t_j} * \mu$ for all $1 \leq j \leq N.$ Set $X'_f = (X_f^1 + \ldots+ X_f^N)/N$. Clearly $X'_f$ can be viewed as the process associated to the measure $\mu'$ and 
$$\mathbb{E}(\sup_{s,t}|X'_f(t) - X'_f(s)|) \leq \frac{1}{N}\sum_{j=1}^N \mathbb{E}(\sup_{s,t}|X^j_f(t) - X^j_f(s)|) = \mathbb{E}(\sup_{s,t}|X_f(t) - X_f(s)|)$$ since each $X^j_f$ has the same distribution as $X_f$ (for all $j$). The comparison principle yields  $\mathscr{P}(\mu) \subset \mathscr{P}(\mu')$. \hfill \qed 
\end{proof}

\paragraph{Remark 1:} If  $\mu$ is the spectral measure of a stationary complex Gaussian process $(\xi_n)$ and $\mu$ has a  density $\varphi$ such that  $\varphi \in L^p(du)$ for some $1 < p < \infty$, then for any $f \in \mathscr{P}(\mu)$ and $t\in \mathbb{T}$, the formal Fourier random series $\sum_{n \in \mathbb{Z}} \hat f(n) \xi_n e^{i n t}$ converges in $L^2(\mathbb{P})$. Indeed, for $t \in \mathbb{T}$ let $U: \mbox{span}(e^{2\pi n t}) \to \mbox{span}(\xi_n)$ be the linear map such that $U(e^{2\pi n i t}) = \xi_n$ for all $n\in \mathbb{Z}$. 
By relation (\ref{sde34rewd2we3}) $U$ defines an isometry from $L^2(\mu)$ to $L^2(\mathbb{P})$. Then by H\"older's inequality $L^r(du) \subset L^2(\mu)$ for some $1 < r < \infty$ and hence $U$ is bounded from $L^r(du)$ to $L^2(\mathbb{P})$. Since the Fourier transform of any $f \in L^r(du)$ converges in $L^r(du)$ it follows that for any $f$ in $L^r(du)$ the series $\sum_{n\in \mathbb{Z}} \hat f(n) \xi_n$ converges in $L^2(\mathbb{P})$. Applying this to the translates $f_t(u) = f(t+u)$ of $f$ yields that the random Fourier series $\sum_{n\in \mathbb{Z}} \hat f(n) \xi_n e^{2\pi n i t}$ converges in $L^2(\mathbb{P})$ for every $t\in \mathbb{T}$.

\section{Banach algebras spanned by the fractional Gaussian noise}
\subsection{Spectral function density of fractional Gaussian noise}
Now we provide explicit examples of Pisier-type Banach algebras $\mathscr{P}(\mu)$. 
The classical discrete complex fractional Gaussian noise (fGn) of parameter $0 \leq H < 1$ is the complex  Gaussian process $\Delta = (\Delta_n)_{n\in \mathbb{Z}}$ with mean $0$ and covariance function:
      \begin{eqnarray} \label{ew32wswaws2}
\mathbb{E}(\Delta_n \overline{\Delta_m}) = \gamma(n-m) =  {\scriptstyle\frac{1}{2}}|n-m+1|^{2H} + {\scriptstyle\frac{1}{2}}|n-m-1|^{2H} -|n-m|^{2H},\,\,\,m,n \in \mathbb{Z}. 
\end{eqnarray}
The process $\Delta$ can be realised as $\Delta_n = (\delta^1_n + i \delta^2_n)/\sqrt{2}$ where $\delta^1 = (\delta^1_n)$  is the classical real fractional Gaussian noise and $\delta^2$ is an independent copy of $\delta^1$. (Here $\mathbb{E}(\delta^1_n \delta^1_m) = \gamma(n-m)$.)
The process $\Delta$ reduces to a sequence of i.i.d complex Gaussian random variables for $H = 1/2$ but generally for $H \ne 1/2$ the random variables $(\Delta_n)$ are {\it interdependent}. (Fractional Gaussian noise is an important model with wide range applications in physics and signal processing.) It is well-known that the sequence $\Delta$ has a spectral density function and we shall denote it by $\varphi_H$. The covariance function (\ref{ew32wswaws2}) implies that for any $\ell \in \mathbb{Z}$, the sequence 
$(\Delta_{\ell+k})_{k \in \mathbb{Z}}$ has the same distribution as the initial sequence $\Delta = (\Delta_{k})_{k \in \mathbb{Z}}$, that is the distribution of $\Delta$ is stationary. 
Moreover for any $\ell \in \mathbb{Z}$, the sequence $\Delta^{\ell} = (\Delta^{\ell}_k)_{k \in \mathbb{Z}}$ defined by
          \begin{eqnarray*}
\Delta^{\ell}_k = \frac{1}{\ell^{H}}\left(\sum_{ k \ell \leq j <(k+1)\ell} \Delta_j\right),\,\, \mbox{ for all } k \in \mathbb{Z}
\end{eqnarray*}
has the same distribution as the initial sequence $\Delta =  (\Delta_{k})_{k \in \mathbb{Z}}$. This means that the distribution of $\Delta$ is self-similar. (See Sinai \cite{Sinai} for some details on self-similar distributions.)  This can be seen by noting that the fGn $(\Delta_k)$ is a sequence of the increments $(X(k+1) - X(k))$ of the fractional Brownian motion of index $H$ and it is well-known that it is a self-similar process in the sense that, for any $\alpha>0$,  
 the processes $\{X(\alpha t): t\in \mathbb{R}\}$ has the same distribution as the process $\{\alpha^{H} X(t): t\in \mathbb{R}\}$.  Indeed,        
$$\Delta^{\ell}_k = \frac{1}{\ell^{H}} \left(X((k+1)\ell  - X(k \ell) \right) \sim X(k+1) - X(k) = \Delta_k$$ by the self-similarity of fractional Brownian motion.  
Since the distribution of $\Delta$ is stationary and self-similar, then by a result of Sinai \cite[Theorem 2.1]{Sinai}, it admits a spectral density function $\varphi_H$ on the unit circle $\mathbb{T}$ given by
  \begin{eqnarray} \label{dsdfesw23}
\varphi_H(t) = C(H)|e^{2 \pi i t}-1|^2\left(\sum_{n=-\infty}^\infty \frac{1}{|t+n|^{2H+1}}\right),\, t \in \mathbb{T}, t \ne 0
\end{eqnarray}
where $C(H)$ is a normalising constant. 
Clearly, 
\begin{eqnarray} \label{dsdfesw2312s}
\varphi_H(t) &=& 4 C(H) \left(\sin^2 \pi t\right) \sum_{n=0}^\infty \left(\frac{1}{(n+t)^{2H+1}} + \frac{1}{(n+ 1-t)^{2H+1}}\right) \nonumber\\
           & = &4 C(H) \left(\sin^2 \pi t\right) (\zeta(2H+1, t) + \zeta(2H+1, 1-t))
\end{eqnarray}
 where $\zeta(.,.)$ is the classical Hurwitz zeta function.
For $H = 1/2$, obviously $\varphi_H = 1$ (constant function on $\mathbb{T}$). For $H \ne 1/2$, $\varphi_H$ is clearly continuous on $(0, 1)$. Moreover it is clear 
that 
  $\zeta(2H+1, t) = O(t^{-(2H+1)})$  for $t$ near $0$ and $\zeta(2H+1, 1-t) = O(1-t)^{-(2H+1)}$  for $ t$ near $1$. Since also 
      $\sin^2 \pi t = O(t^2 (1-t)^2)$ for $t$ near $0$ or $1$, it follows from (\ref{dsdfesw2312s}) that  
\begin{eqnarray} \label{sddsdwdwwsas}
\varphi_H(t)  = O(t^{1-2H} (1-t)^{1-2H}),\,\, \mbox{ for } t \mbox{ near } 0 \mbox{ or } 1.
\end{eqnarray}
This implies that for $H\leq 1/2$, the function $\varphi_H$ is  continuous on the whole unit circle $\mathbb{T}$ and for $H > 1/2$, $t = 0$ and $t = 1$ are singular points (actually $0 = 1$ in $\mathbb{T}$). However even in the case where $H > 1/2$,   $\varphi_H$ is integrable and in fact (\ref{sddsdwdwwsas}) implies that 
      \begin{eqnarray} \label{dsdsedsqw3e}
      \varphi_H \in L^{p}(\mathbb{T})\mbox{  for all } p < 1/(2H-1),\,\, H>1/2
\end{eqnarray}
         
 \subsection{The Banach algebra $\mathscr{P}(H)$}
The class $\mathscr{P}(\mu_{\Delta})$ where $\mu_\Delta$ is the spectral measure of the discrete fractional Gaussian noise $\Delta = (\Delta_n)$ of parameter $H$ will simply be denoted $\mathscr{P}(H)$. Its study was initiated in \cite{Mukeru_LMJ} 
where it was proven that for $0 \leq H \leq 1/2$, $ \mathscr{P}(\frac{1}{2}) \subset \mathscr{P}(H)$ and it was conjectured that $\mathscr{P}(H) \subset \mathscr{P}(\frac{1}{2})$ for $\frac{1}{2} \leq H < 1.$ 
The question whether $\mathscr{P}(H)$ is in general a Banach algebra with the usual multiplication of functions was also raised. With the discussion above, it is now clear that for all $0 \leq H < 1$,  $\mathscr{P}(H)$ is a homogeneous Banach algebra on which all 1-Lipschitzian functions operate.  
Moreover for $0 \leq H < 1/2$,  the continuity and boundedness of $\varphi_H$ and the fact that $\varphi_H(t) > 0$ everywhere on $\mathbb{T}$ (except for $t = 0= 1$) implies by Theorem \ref{referee_c1}  that $\mathscr{P}(H) = \mathscr{P}(\frac{1}{2}) = \mathscr{P}$. 
Finally for $H$ varying in the interval $[0, 1)$, $\mathscr{P}(H)$ is a decreasing family of homogeneous Banach algebra. That is,
\begin{theorem}
 For any $0 \leq H_1 \leq H_2 < 1$, 
            $\mathscr{P}(H_2) \subset \mathscr{P}(H_1).$                      
\end{theorem}
\begin{proof}
This follows from the structure of the spectral density function $\varphi_H$ of the fractional Gaussian noise: 
For for $H_1 \leq  H_2$, it is the case that the quotient $\varphi_{H_1}/\varphi_{H_2}$ is bounded. Indeed, 
write
 $$\frac{\varphi_{H_1}}{\varphi_{H_2}} = K\, \frac{(\zeta(1+2H_1, t) + \zeta(1+2H_1, 1-t))}{(\zeta(1+2H_2, t) + \zeta(1+2H_2, 1-t))},\,\, t \in (0, 1) $$
with $K = C(H_1)/C(H_2)$. Clearly $\varphi_{H_1}/\varphi_{H_2}$  is continuous in the open interval $(0, 1)$.
 For $t$ in a neighborhood of $0$ in $\mathbb{T}$, 
 $$\frac{\varphi_{H_1}}{\varphi_{H_2}} = O(t^{-(1+ 2H_1)}/t^{-(1+2H_2)}) = O(t^{2(H_2- H_1)}),$$ which implies that  $\lim_{t \to 0} \varphi_{H_1}(t)/\varphi_{H_2}(t)= 0.$
Similarly for $t$ near $1$, 
      $\lim_{t \to 1} \varphi_{H_1}(t)/\varphi_{H_2}(t) = 0.$ This implies that $\varphi_{H_1}/\varphi_{H_2}$ is continuous in the closed interval $[0, 1]$ and hence $\varphi_{H_1}/\varphi_{H_2}$  is continuous and bounded on the unit circle $\mathbb{T}$. Then for a constant 
  $M$ depending only on $H_1$ and $H_2$ we have that 
    \begin{eqnarray} \label{sdsdfssaw223e}
\varphi_{H_1}(t) \leq M\, \varphi_{H_2}(t) \mbox{ for all } t \in \mathbb{T}.
\end{eqnarray}
Then for each $f \in C(\mathbb{T})$, 
      $$\int_{\mathbb{T}} |f(t+u) - f(s+u)|^2 \varphi_{H_1}(u) du \leq M \int_{\mathbb{T}} |f(t+u) - f(s+u)|^2 \varphi_{H_2}(u) du$$
      which yields by the comparison principle that $\mathscr{P}(H_2) \subset \mathscr{P}(H_1)$. 
\end{proof}

\section{Boundedness of random Fourier series}

\begin{theorem}\label{MuPi}
Let $\xi = (\xi_n)_{n\in \mathbb{Z}}$ be a stationary complex Gaussian process with mean $0$, unit variance, covariance matrix $(\gamma(n-m))_{n,m \in \mathbb{Z}}$ and such that $\mathbb{E}(\xi_n \xi_m) = 0$ for all $n,m \in \mathbb{Z}$. Assume that there exists a constant $b\geq 0$ such that the operator defined in $\ell^2(\mathbb{Z})$ by the bi-infinite matrix $(a_{nm})_{n,m\in \mathbb{Z}}$ given by
 $$a_{nm} = |\gamma(n-m)| |n m|^{-b} \mbox{ for } n\ne 0,\,m\ne 0 \mbox{ and } a_{nm} = 0 \mbox{ for } n=m=0$$
is bounded. Let $f \in L^2(\mathbb{T})$. If 
\begin{eqnarray}\label{dsdsdefre}
\sum_{|n| \geq 2} \frac{\left(\sum_{|k| \geq |n|} |\hat f(k)|^2 |k|^{2 b}\right)^{1/2}}{|n| (\log |n|)^{1/2}} < \infty,
\end{eqnarray}
then the random series $t \mapsto \sum_{n\in \mathbb{Z}} \hat f(n) \xi_n e^{2\pi n i t}$ is bounded almost surely. 
\end{theorem}
In particular if $f\in C(\mathbb{T})$ and satisfies (\ref{dsdsdefre}) then $f \in \mathscr{P}(\mu)$ where $\mu$ is the spectral measure of $\xi$.
\begin{proof}
Condition (\ref{dsdsdefre}) is a generalisation of the following condition for independent variables:
For a sequence $(g_n)$ of independent standard Gaussian random variables and for $f \in C(\mathbb{T})$ and a sequence $(a_n)_{n\geq 1}$ of real numbers, if
  $$\sum_{n=2}^\infty \frac{\left(\sum_{k=n}^\infty a_k^2\right)^{1/2}}{n (\log n)^{1/2}} < \infty,$$
  then the random series    $t \mapsto \sum_{n=1}^\infty a_n g_n e^{2\pi n i t}$, $t\in \mathbb{T}$, converges uniformly. (See Marcus and Pisier \cite[p. 122]{Marcus_Pisier_book}.) 
Consider the series $S(f,\xi,t) =  \sum_{n\in \mathbb{Z}} \hat f(n) \xi_n \exp(2\pi i n t)$, $t \in \mathbb{T}$ and write for $N \in \mathbb{N}$
     $$S_N(t) = \sum_{|n| \leq N} \hat f(n) \xi_n e_n(t),\,\,\mbox{ with } e_n(t) = \exp(2\pi i n t),\,\, \mbox{ for all } n $$
     and  $\gamma(n-m) = \mathbb{E}(\xi_n \overline{\xi_m})$.
Then for $N \leq M$ in $\mathbb{N}$, 
 \begin{eqnarray*} 
\mathbb{E}(|S_M(t) -S_{N}(t)|^2) & \leq & \sum_{N< |n| \leq M} |\hat f(n)|^2 |e_n(t)|^2 + 
\sum_{N < |n|\ne |m| \leq M} |\gamma(n-m)| |\hat f(n)| |\hat f(m)| |e_n(t)| |e_m(t)|\\
& \leq & \sum_{N< |n| \leq M} |\hat f(n)|^2 |e_n(t)|^2 + 
\sum_{N < |n|\ne |m| \leq M} a_{nm} |n|^b |m|^b |\hat f(n)| |\hat f(m)| |e_n(t)| |e_m(t)|.
\end{eqnarray*}
Now since the matrix $A = (a_{nm})$ defines a bounded operator in $\ell^2(\mathbb{Z})$, then
 $$\sum_{N < |n|\ne |m| \leq M} a_{nm} |n m|^{b} |\hat f(n)| |\overline{\hat f(m)} |e_n(t)| |e_m(t)|  \leq K \sum_{N < n \leq M} |\hat f(n)|^2 |n|^{2b} |e_n(t))|^2$$
 for some constant $K>0$. (In fact we may write this sum as $w A \overline{w}$ where $w = (w_n)_{n\in \mathbb{Z}}$ and $w_n = |\hat f(n)| |n|^b |e_n(t)|$ for $N < n \leq M$ and $w_n = 0$ outside the interval $(N, M]$ in $\mathbb{Z}$ and since $A$ is  positive semidefinite and bounded, then $w A \overline{w} \leq K \|w\|^2$.)
It follows that   
  $$\mathbb{E}(|S_M(t) -S_{N}(t)|^2) \leq K' \sum_{N < n \leq M} |\hat f(n)|^2 |n|^{2b} |e_n(t))|^2$$ for some constant $K'>0$.
Now clearly inequality (\ref{dsdsdefre}) implies in particular that for each $t \in \mathbb{T}$ the series $\sum_{ n\in \mathbb{Z}} |\hat f(n)|^2 |n|^{2b} |e_n(t)|^2$ converges from which it follows that 
$\lim_{N, M \to \infty} \mathbb{E}(|S_M(t) -S_{N}(t)|^2) = 0$. This yields the $L^2$--convergence of the series $S(f, \xi, t)$. The same argument yields that for all $t,s \in \mathbb{T}$, 
        $$\mathbb{E}|S(f, \xi, t)- S(f, \xi, s)|^2  \leq K  \sum_{n\in \mathbb{Z}} |\hat f(n)|^2 |n|^{2b} |e_n(t)- e_n(s)|^2$$  for some constant $K>0$.
That is 
     $$\mathbb{E}|S(f, \xi, t)- S(f, \xi, s)|^2  \leq K \mathbb{E}\left|\sum_{n\in \mathbb{Z}} \hat f(n) |n|^{b} g_n (e_n(t)- e_n(s))\right|^2$$
     where $(g_n)$ is the standard complex Gaussian sequence. 
Inequality (\ref{dsdsdefre}) implies that the random series 
 $t \mapsto \sum_{n\in \mathbb{Z}} \hat f(n) |n|^{b} g_n e_n(t)$ converges uniformly on $\mathbb{T}$. This yields by the comparison principle that the paths $t \to S(f, \xi, t)$ are bounded almost surely. If in addition $f$ is continuous then $f \in \mathscr{P}(\xi)$. \quad \hfill \hfill \qed
\end{proof}              
An example of a covariance function $\gamma$ for which the matrix $(\gamma(n-m) |n m|^{-b})$ defined a bounded operator in $\ell^2(\mathbb{Z})$ is obtained by assuming
        $|\gamma(n-m)| \leq K |n-m|^{-a}$ 
for some constants $K>0$ and $0 < a < 1.$ 
This can be obtained by using the classical Schur test together with the following well-known fact: for real numbers $0 <a < 1$ and $ d >  0$ such that $a + d >1$, there exists a constant $K> 0$ (depending on $a$ and $d$) such that 
   $$\mbox{ for all } n \in \mathbb{Z}, \sum_{m \ne n}|n-m|^{-a} |m|^{-d} \leq K |n|^{1-(a+d)}.$$
Set $a_{nm} = |n-m|^{-a} |n m |^{-b}$ (with $a_{nn} = a_{n,0} = a_{0,n} = 0$)
and consider a real number $c>0$ such that $a + b + c > 1$ and $x = (x_n)_{n\in \mathbb{Z}}$ with  $x_n = |n|^{-c}$ for $ \ne 0$ and $x_0 = 1$. It is the case that 
    $$\forall  n \,\,\, \sum_{m \in \mathbb{Z}} a_{nm} x_m  \leq K x_n \,\,\, \mbox{ and }\,\,\,   \forall m \,\,\, \sum_{n \in \mathbb{Z}} x_n a_{nm} \leq K  x_m.$$ 
Hence Schur's test yields that the operator defined by $(a_{nm})$ is bounded.  In particular the covariance function $\gamma_H$ of the fractional Gaussian noise of index $H$ is such that the corresponding matrix $(a_{nm})$ given by $a_{nm} = \gamma_H(n-m) |nm|^{-b}$ for any $b>0$
 defines a bounded operator. Indeed it is easy to verify that  $0 < \gamma_H(n - m) \leq |n - m|^{-2(1-H)}$ for $m  \ne n.$      
\begin{corollary}
Let $f\in L^2(\mathbb{T})$ and $\frac{1}{2} \leq H < 1$. If there exist $\alpha > H$ and $K>0$ such that
$$|\hat f(n)| \leq K |n|^{-\alpha},\,\, \mbox{ for all } n \ne 0,$$
then the function $t \mapsto \sum_{n \in \mathbb{Z}} \hat f(n) \Delta_n e^{2 \pi  i n t}$ is bounded almost surely.  In particular if in addition $f\in C(\mathbb{T})$ but the series $\sum_{n\in \mathbb{Z}} |\hat f(n)|$ diverges, then $f\in \mathscr{P}(H)$ but $f\notin A(\mathbb{T}).$
\end{corollary}
This follows from the fact that by taking $b = H - 1/2 > 0$ it is clear that inequality (\ref{dsdsdefre}) holds.

\paragraph{Remark 2:} It would be interesting to investigate whether inequality (\ref{dsdsdefre}) implies the  uniform convergence (almost surely) of the random Fourier series $t \to \sum_{n\in \mathbb{Z}} \hat f(n) \xi_n e^{2\pi n i t}$ as it is the case for i.i.d random variables. Pisier \cite{Pisier_2} obtained that the Banach space of $f \in L^2(\mathbb{T})$ such that  $\sum_{n \in \mathbb{Z}} Z_n \hat f(n) e^{2\pi i n t}$ (with $(Z_n)$ i.i.d Gaussian variables) is bounded a.s. is of cotype 2. It would be interesting to consider this question for the general stationary Gaussian process $(\xi_n)$. Pedersen \cite{Pedersen} proved that spectral synthesis holds in $\mathscr{P}$. This question is also of some interest in the general setting discussed in this paper. 
\paragraph{Acknowledgments:} I would like to thank the anonymous referee for very helpful comments, suggestions and new results that have significantly improved this paper. Theorem 2 was directly suggested by the referee.
This research received with gratitude funding from the College of Economics and Management Sciences of the University of South Africa and the EU project CID. Most of this work was conducted during a visit to the National Institute for Research in Digital Science and Technology (INRIA) in Nancy. I thank Mathieu Hoyrup for the invitation and his hospitality.

\end{document}